\documentclass[11pt]{amsart}
\usepackage{amsmath, amssymb, amsthm,  epsfig}
\theoremstyle{plain}
\newtheorem{thrm}{Theorem}[section]
\newtheorem*{thrm*}{Theorem}
\newtheorem{lemma}[thrm]{Lemma}
\newtheorem{prop}[thrm]{Proposition}
\newtheorem{cor}[thrm]{Corollary}
\theoremstyle{definition}
\newtheorem{dfn}[thrm]{Definition}
\theoremstyle{remark}
\newtheorem{rmrk}[thrm]{Remark}
\theoremstyle{example}

\numberwithin{equation}{section}
\usepackage{color}

\begin{document}

\newcommand{\tx}{\tilde x}
\newcommand{\R}{\mathbb R}
\newcommand{\N}{\mathbb N}
\newcommand{\C}{\mathbb C}
\newcommand{\lie}{\mathcal G}
\newcommand{\hN}{\mathcal N}
\newcommand{\D}{\mathcal D}
\newcommand{\A}{\mathcal A}
\newcommand{\B}{\mathcal B}
\newcommand{\sL}{\mathcal L}
\newcommand{\sLi}{\mathcal L_{\infty}}

\newcommand{\G}{\Gamma}
\newcommand{\x}{\xi}

\newcommand{\eps}{\epsilon}
\newcommand{\al}{\alpha}
\newcommand{\be}{\beta}
\newcommand{\p}{\partial}  
\newcommand{\lig}{\mathfrak}

\def\dist{\mathop{\varrho}\nolimits}

\newcommand{\BCH}{\operatorname{BCH}\nolimits}
\newcommand{\Lip}{\operatorname{Lip}\nolimits}
\newcommand{\Hol}{C}                             
\newcommand{\lip}{\operatorname{lip}\nolimits}
\newcommand{\capQ}{\operatorname{Cap}\nolimits_Q}
\newcommand{\pCap}{\operatorname{Cap}\nolimits_p}
\newcommand{\Om}{\Omega}
\newcommand{\om}{\omega}
\newcommand{\half}{\frac{1}{2}}
\newcommand{\e}{\epsilon}
\newcommand{\vn}{\vec{n}}
\newcommand{\X}{\Xi}
\newcommand{\tLip}{\tilde  Lip}

\newcommand{\Span}{\operatorname{span}}

\newcommand{\ad}{\operatorname{ad}}
\newcommand{\Hm}{\mathbb H^m}
\newcommand{\Hn}{\mathbb H^n}
\newcommand{\Hone}{\mathbb H^1}
\newcommand{\Lie}{\mathfrak}
\newcommand{\Layer}{V}
\newcommand{\hgrad}{\nabla_{\!H}}
\newcommand{\im}{\textbf{i}}
\newcommand{\nz}{\nabla_0}
\newcommand{\s}{\sigma}
\newcommand{\se}{\sigma_\e}

\newcommand{\ued}{u^{\e,\delta}}
\newcommand{\ueds}{u^{\e,\delta,\sigma}}
\newcommand{\tnabla}{\tilde{\nabla}}

\newcommand{\bx}{\bar x}
\newcommand{\by}{\bar y}
\newcommand{\bt}{\bar t}
\newcommand{\bs}{\bar s}
\newcommand{\bz}{\bar z}
\newcommand{\btau}{\bar \tau}
\newcommand{\bY}{\bar Y^{\e}}
\newcommand{\bd}{\bar{d}}

\newcommand{\LC}{\mbox{\boldmath $\nabla$}}
\newcommand{\Ne}{\mbox{\boldmath $n^\e$}}
\newcommand{\nuo}{\mbox{\boldmath $n^0$}}
\newcommand{\nuu}{\mbox{\boldmath $n^1$}}
\newcommand{\nue}{\mbox{\boldmath $n^\e$}}
\newcommand{\nuek}{\mbox{\boldmath $n^{\e_k}$}}
\newcommand{\dse}{\nabla^{H\Su, \e}}
\newcommand{\dso}{\nabla^{H\Su, 0}}
\newcommand{\tX}{\tilde X}

\newcommand{\Xie}{X^\epsilon_i}
\newcommand{\Xje}{X^\epsilon_j}
\newcommand{\Su}{\mathcal S}
\newcommand{\F}{\mathcal F}

\title[Harnack Inequality]{A  subelliptic analogue of Aronson-Serrin's Harnack inequality}

\author{Luca Capogna}\address{Institute for Mathematics and its Applications, University of Minnesota, Minneapolis, MN 55455 \\ Department of Mathematical Sciences,
University of Arkansas, Fayetteville, AR 72701}\email{lcapogna@uark.edu}
%
\author{Giovanna Citti}\address{Dipartimento di Matematica, Piazza Porta S. Donato 5,
40126 Bologna, Italy}\email{citti@dm.unibo.it}
\author{Garrett Rea}\address{
Mathematics Department,
Missouri Southern State University,
Joplin, MO 648010 }\email{rea-garrett@mssu.edu}
\keywords{Harnack inequality, Quasilinear parabolic PDE, homogeneous spaces, Carnot-Caratheodory distance\\
LC is partially funded by NSF award  DMS 1101478, GC is partially funded by the European project CG-DICE }

\begin{abstract} We study the Harnack inequality for weak solutions of  a class of degenerate parabolic quasilinear PDE
$$\p_t u={-}X_i^* A_i(x,t,u,Xu)+ B(x,t,u,Xu),$$ 
in cylinders $\Om\times (0,T)$ where  $\Om \subset M$  is an open subset of a   manifold $M$ endowed with  control metric $d$ corresponding to  a system of Lipschitz continuous vector fields $X=(X_1,...,X_m)$ and  a measure $d\sigma$. We show that the Harnack inequality follows from
the basic hypothesis of doubling condition and a  weak Poincar\'e inequality in the metric measure space $(M,d,d\sigma)$. We also show that such hypothesis hold for a class of Riemannian metrics $g_\e$  collapsing to a sub-Riemannian metric $\lim_{\e\to 0} g_\e=g_0$ uniformly in the parameter $\e\ge 0$.
 \end{abstract}
\maketitle

\setcounter{tocdepth}{2}
\section{Introduction}
In two important works,  Saloff-Coste \cite{SC} and Grigor'yan \cite{grig} established the equivalence between a Harnack inequality for weak solutions of  a class of linear,  smooth coefficients PDE, with non-negative  symbol, and two key metric-measure properties of the ambient space: (1) a doubling inequality for balls in a control metric naturally associated to the operator and (2)  a Poincar\'e inequality involving a notion of gradient naturally associated to the operator. The Saloff-Coste-Grygor'yan results were later extended and applied to a broad range of PDE and variational problems (e.g.  \cite{MR1359957},  and references therein). One of the motivating examples where such theory  can be applied consists in a class of  subelliptic operators associated to a family of H\"ormander vector fields $X=(X_1,...,X_m)$ and their Carnot-Caratheodory distance $d(\cdot, \cdot)$ (see \cite{Montgomery:book}, \cite{NSW} and  \cite{SCS1991}).

The purpose of this note is twofold:

(a)  We  show that doubling property and  Poincar\'e inequalities imply the Harnack inequality for weak solutions of  the quasilinear, degenerate parabolic PDE \eqref{equation}. Viceversa, we note that if the Harnack inequality holds for every operator in this class, then by virtue of \cite{SC}, \cite{grig} both the doubling property and the  Poincar\'e inequalities hold. Our work represents an analogue, in the subelliptic setting, of  the work of Aronson and Serrin in \cite{MR0244638}.

(b) We want to show that the constants arising in the doubling and Poincar\'e inequalities are stable with respect to a well known and much used, Riemannian approximation scheme, in which the Carnot-Carath\'{e}odory metric is approximated (in the Gromov-Hausdorff sense) by a family of degenerating Riemannian metrics. Such approximations have been studied, for instance, in \cite{Montgomery:book} and \cite{monti-tesi} (where more references can be found).


The main  motivation for our work is to provide the necessary background  to study existence and regularity for a large class of non-linear degenerate parabolic PDE by approximation with solutions of ad-hoc {\it 'regularized'} versions of such PDE, in what is
occasionally called  the {\it vanishing viscosity} approach. Part (b) shows that the homogenous structure and the constants in the Poincar\'{e} inequality associated to the geometry underlying the approximating PDE are stable, so that the results in part (a) yield uniform estimates on the H\"older regularity of the solutions, in the maximum principle and in the Harnack inequalities.

Our proofs rest on a careful extension of ideas and techniques developed in the works \cite{MR0244638},  \cite{SC}, \cite{grig} and  \cite{NSW}. The arguments are  technically involved and rest on results spread through a large number of papers. Part (a) has been largely developed by the third named author (GR) in his
PhD dissertation \cite{rea}. While completing the final version of the present paper we were made aware of
recent, pertinent and  very interesting work of Kinnunen, Marola, Miranda and Paronetto \cite{KMMP} where an alternative approach
to part (a), in a more general setting, is studied. The authors in \cite{KMMP} derive Harnack inequalities out of membership
in the so-called parabolic De Giorgi classes, in the general context of doubling metric measure spaces endowed with a
Poincar\'{e} inequality. It is not clear what is the most general setting in which a regularity theory for parabolic PDE would make sense and could be developed. This problem is intimately connected to finding the most general setting for a first order calculus in metric measure spaces \cite{heinonen:calc}. Among other pertinent references we mention:  the result of   Kinnunen and Kuusi   \cite{MR2274548} where it is proved that doubling and (1,p)-Poincar\'{e} inequality
imply the parabolic Harnack inequality for a class of doubly nonlinear
equation of p-Laplacian type with weights, and the work of Sturm, where the Saloff-Coste and Grigor'yan results are extended to the general setting of local Dirichlet spaces (see for instance \cite{MR1387522}).

\subsection{\bf Analogues of Aronson-Serrin's Harnack inequality and maximum principle} Consider a distribution of  Lipschitz continuous vector fields  $X=\{X_1,...,X_m\}$  in a  bounded open set $\Om\subset \R^n$. We denote by  $X_i^*$ their adjoint,   and by $d(\cdot, \cdot):\Om\times \Om\to \R^+$ the control distance associated to $X$. In this paper we will assume that $d$ satisfies a doubling condition\footnote{See next section for the pertinent definitions} \eqref{doubling} w.r.t. Lebesgue measure and with doubling constant $C_D$. Correspondingly one finds a number $N>0$ that acts as {\it homogeneous dimension}, i.e. volume of metric balls grows at least like a power of $N$ of their radius (Proposition \ref{longprop}(i)). In this paper we will always assume $N>2$.

We will show that if one further assumes as  hypotheses a   Poincar\'{e} inequality \eqref{Poincare} with constant $C_P$ and the a.e. differentiability of the metric \eqref{Lip} with Lipschitz constant $C_L$, then one has  an analogue of Aronson and Serrin's results  from \cite{MR0244638}. Specifically, we establish a maximum principle and a Harnack inequality for weak solutions of the quasilinear degenerate parabolic PDE
\begin{equation}\label{equation}
\p_t u = -X_i^* A_i(x,t,u,Xu)+ B(x,t,u,Xu),
\end{equation}
where $A_i$ and $B$ satisfy the structure conditions \eqref{structure}.

\begin{thrm}[Maximum Principle]\label{T:max}
Let $u$ be a weak solution of \eqref{equation} in the parabolic cylinder $Q=\Om\times (0,T)$ and assume that there exists $M\in\R$  such that for every $\delta>0$
one has $u\le M+\delta$ in a neighborhood of the parabolic boundary $\p_p Q=(\Om\times\{t=0\})\cup( \p\Om\times (0,t))$.
There exists a positive constant $C$ depending only on $diam(\Om)$, $C_D$, $C_L$, $C_P$,  the structure conditions \eqref{structure} and on $Q$ such that for a.e. $x\in Q$ one has \begin{equation}\label{max} u(x)\le M+C\kappa,\end{equation}
where $\kappa =(||b||+||d|)|M|+||f||+||g||.$
\end{thrm}

\begin{thrm}[Harnack Inequality]\label{T:harnack}
Let $u\ge 0$ be a weak solution of \eqref{equation} in $Q$ and $0<\rho<20R$ with $R$ as in \eqref{doubling}. If one has a  subcylinder $Q_{3\rho}=B(\bar x,3\rho)\times(\bar t-9\rho^2,\bar t)\subset Q$ then there exists a constant $C>0$ depending on
$C_D$, $C_L$, $C_P$, the structure conditions \eqref{structure} and on $\rho$ such that
\begin{equation}
\max_{Q^-}u\le C\min_{Q^+} (u+\rho^\theta k),
\end{equation}
where $$Q^+=B(x,\rho)\times(\bar t-\rho^2,\bar t)\text{  and }Q^-=B(x,\rho)\times(\bar t- 8\rho^2, \bar t-7 \rho^2)$$
$\theta>0$ is defined as in \eqref{theta}, and we have let $k=||f||+||g||+||h||$.

\end{thrm}

\subsection{\bf Stability of the homogenous structure and of the Poincar\'e inequality in the Riemannian approximation scheme to a sub-Riemannian metric}
We consider examples of vector fields satisfying all the
assumptions (D), (L) (P). In view of the seminal work in \cite{NSW} and \cite{jer:poincare} it is clear that these assumptions are
satisfied by smooth distributions of  vector fields
$X=(X_1,\ldots ,X_m)$ in $R^n$ satisfying the H\"ormander finite rank condition
\begin{equation}    \label{Hor} rank \,\, Lie
(X_1,\ldots ,X_m)(x)=n, \quad \forall x\in \Omega.\end{equation}
Consequently  weak solutions of the equation
\begin{equation} 
\p_t u = -X_i^* A_i(x,t,u,Xu)+ B(x,t,u,Xu),
\end{equation}
written in terms of these vector fields satisfy Harnack's inequality and are H\"older continuous with respect to the Carnot-Carath\'{e}odory metric associated to $X$.

\bigskip

In many applications (see for instance \cite{Pauls:minimal}, \cite{ccm1}, \cite{CCM2}, \cite{cittimanfredini:uniform} and \cite{chy}) one is interested in approximating the subelliptic PDE with a sequence of elliptic regularizations, so that the corresponding solutions converge to {\it viscosity solutions} of the original PDE. From the point of view of the metric background this approximation corresponds to studying {\it tame} approximations of the sub-Riemannian metric (and the associated  distance function $d_0(\cdot, \cdot) $) with
a one-parameter family of degenerating Riemannian metric (associated to distance functions $d_\e(\cdot, \cdot)$), which converge in the Gromov-Hausdorff sense as $\e\to 0$ to the original one. Such approximations
have been studied, for instance, in \cite{Montgomery:book} and \cite{monti-tesi}.
  A viscosity approximation of a time independent equation was considered in \cite{CiMa-F} where uniform estimates for the fundamental solution  were established.
Here we further develop these results, applying it to the time dependent equation, and proving that also the
constants in the Doubling and  Poincar\'{e} inequalities can be chosen independently of the parameter $\epsilon,$
thus leading to uniform H\"older regularity of weak solutions of the associated approximating equations.

To show a simple, almost na\"{i}ve, example of such approximation as well as an application of our work, we consider the Heisenberg group $\Hone$ which can be identified with $R^3$ with a set of left invariant horizontal vector fields $X_1,X_2$ such that $X_1,X_2, X_3=[X_1,X_2]$ generates the whole Lie algebra (see \cite{ste:harmonic} for a detailed description). For $\e\ge 0$ we consider a family of left-invariant Riemannian metrics $g_\e$ in $\Hone$ defined so that the frame $X_1, X_2, \e X_3$ is orthonormal.
A well known technique to study solutions of the degenerate parabolic PDE
(to mention one example out of many)
\begin{equation}\label{heat}\p_t u= X_1^2 u +X_2^2 u\end{equation}
is to consider a family of solutions $u^\e$ of the strongly parabolic PDE
\begin{equation}\label{heate}\p_t u^\e = X_1^2 u^\e + X_2^2 u^\e + \e^2 X_3^2 u^\e\end{equation}
with fixed ($\e-$independent) data on the boundary of given a bounded parabolic cylinder in $\Hone\times \R^+$. The PDE  \eqref{heate} is the heat equation for the Riemannian metric $g_\e$, while \eqref{heat} is the {\it subelliptic} heat equation for
the sub-Riemannian metric $g_0$. Parabolic theory yields existence and uniqueness of smooth solutions $u^\e$ to \eqref{heate}. The problem is that such smoothness may degenerate as $\e\to 0$. In this paper we show (see Proposition \ref{indipendenza}) that the doubling property and the Poincar\'{e} inequality corresponding to the Riemannian metrics $g_\e$ are stable as $\e\to 0$ and as a consequence of the  {\it stable} Harnack inequality Theorem \ref{T:harnack} the set  $\{u^\e\}_{\e>0}$ is a normal family, with uniform limit $u_0$. Applying the same argument to right invariant derivatives of the solutions (see for instance \cite{CC}) we can easily prove that $u_0$ is smooth and solves \eqref{heat}. Obviously, the linear structure of \eqref{heat} provides far more effective ways of proving existence of smooth solutions, however such methods
cannot be  applied for non-linear PDE, while the techniques in the present paper are purposely designed to be used in the non-linear setting.

The core of  argument  in the stability proof consists in a careful extension of  the arguments in \cite{NSW} to include the additional parameter $\e$,
and show  $\e-$independent bounds on the Jacobian of an exponential mapping related to commutators of the H\"ormander vector fields.
Once such bounds are established,  the doubling condition follows immediately
and the Poincar\'{e} inequality is a consequence of  the work in \cite{LanMor}, \cite{GarofaloNhieu:lip1998} and \cite{frss:embedding}.  To better appreciate the  stability results, we note that even in the simple example described above the geometry of the space changes  radically as $\e\to 0$. In this approximation the curvatures tensors blow up, and the shape and volume of the metric balls changes from Riemannian to sub-Riemannian.  Yet, the homogeneous structure, as well as the constants associated to the Poincar\'e inequality  remain the same.

{\it Acknowledgements} The authors would like to thank Juha Kinnunen, Niko Marola and Michele Miranda Jr. for sharing with us helpful comments and pertinent references.

\section{Preliminaries}
Let $X_1,...,X_m$ be Lipschtiz continuous vector fields in a open set $\Om\subset \R^n$. We denote by $d(\cdot, \cdot):\Om\times \Om\to \R^+$ the control distance associated to $X$ and for $x\in\Om$ and $r>0$, by $B(x,r)=\{y\in \R^n| d(x,y)<r\}$ the corresponding metric balls and by $|B(x,r)|$ their Lebesgue measure.  The average of a function $u$ over  a ball $B=B(x,r)$ is denoted by $u_B$.  For $1\le p\le \infty$ we define the Sobolev spaces $W^{1,p}(\Om)=\{u\in L^p(\Om)|\ Xu=(X_1 u,...,X_m u)\in L^p(\Om)\}$ and $W^{1,p}_0$ to be the closure of $C^\infty(\Om)$ in the norm $||u||_{1,p}^p=||u||_p +||Xu||_p$.

In our results we will rely on the following  hypothesis on the background geometry induced by $X$: For every compact $K\subset \Om$ there exists constants
$C_D=C_D(X,K), C_L=C_L(X,K),C_P=C_P(X,K) >0 ,$   and  $R=R(X,K)>0$ such that for every $x\in K$ and $0<r<R$ one has

\begin{equation}\label{doubling}\tag{D}
 |B(x,r)|\ge C_D|B(x,2r)|.
\end{equation}

\begin{equation}\label{Lip}\tag{L}
d(\cdot, x) \text{ is differentiable a.e. in } \Om \text{ and } | |X d(\cdot, x)||_{L^{\infty}(K)} \le C_L \text{ for every } x\in K.
\end{equation}

\begin{equation}\label{Poincare}\tag{P}
 \int_{B(x,r)} |u-u_B|^2 dx \le C_P {r^2}\int_{B(x,2r)} |Xu|^2 dx .
\end{equation}

In the following we say that a constant depends on $C_D,C_L$ and $C_P$ if it is uniformly bounded when $R,C_D,C_L$ and $C_P$  are so.
As a consequence of these hypotheses one has that the metric measure space $(K,d,dx)$ is an homogenous space with a Poincar\'{e} inequality, and as such it enjoys several useful properties listed below.

\begin{prop}\label{longprop} If hypotheses \eqref{doubling},\eqref{Lip}, \eqref{Poincare} hold then for every compact subset $K\subset \Om$, $x\in K$ and $0<r<R$ one has:

\begin{itemize}
\item (i) (Lower bound on volume of balls) There exists  constants $N=N(C_D)>0$ (called {\it homogeneous dimension} of $K$ w.r.t. $(X,d,dx)$) and  a constant $C_G=C_G(C_D)>0$, such that $|B(x,r)|\ge C_G r^N$.
\item (ii) There exists a test function $\phi\in W^{1,\infty}_0(B(x,2r))$ and a constant $C=C(C_L)>0$, such that $\phi=1$ in $B(x,r)$ and $|X\phi|\le C/r$.
\item (iii)  (Weighted Poincar\'e inequality, \cite[Theorem 8.1]{MR1359957}) Let  $\phi:\R^+\to [0,1]$ be a non increasing function with compact support in a finite interval $[0,R)$ and such that (1) $\sqrt{\phi}\in C^{\infty}_0([0,R))$; (2) For $0<r<R$ one has $\phi(r+ \frac{1}{2}\min(R-r,\frac{r}{2}))\ge \al \phi(r)$. Then there exists  $C=C(C_D,C_P,C_L,N,R)>0$ such that
\begin{equation}\label{weighted}\tag{WP}
\int  |u(y)-u_\phi|^2 \phi(d(y,x)) dy \le C{\bar \kappa}^{-2/N} (\int \phi(d(y,x)) dy)^{2/N} \int |Xu|^2(y) \phi(d(y,x)) dy,
\end{equation}
for every function $u\in W^{1,2}(\Om)$ and where we have let ${\bar \kappa}=\inf_{y\in K} |B(y,R)| (2R)^{-\log_2 C_D}>0$ and $u_\phi=\frac{1}{\int \phi(d(y,x)) dy} \int u (y)\phi(d(y,x)) dy$.
\item (iv) (Sobolev Embedding \cite{MR1359957, hak:sobolev2}) There exists $C=C(C_D,C_P,N)>0$ such that  for every $1\le p <N$,
\begin{equation}\label{Sobolev}\tag{S}
\bigg( \int_\Omega
|u|^{\frac{Np}{N-p}} dx
\bigg)^{\frac{N-p}{Np}}
 \le C  \bigg( \int_{\Omega} |Xu|^p dx \bigg)^{\frac{1}{p}}
\end{equation}
for every function $u\in W^{1,p}_0(\Om)$.
\end{itemize}

\end{prop}
We also notice that in view of Cheeger's Rademacher's theorem \cite{ch:lipschitz}, the differentiability hypothesis \eqref{Lip} holds as a consequence of \eqref{doubling} and \eqref{Poincare}. In the setting of metrics generated by families of Lipschitz vector fields see
 also \cite[Theorem 1.5]{GarofaloNhieu:lip1998} and references therein.

Next, we consider a parabolic cylinder $Q=\Om\times (0,T)$ and address the parabolic BMO spaces and the John-Nirenberg lemma for the homogenous space $(Q,d_p,dx)$, where $d_p$ denotes  the parabolic distance $d_p((x,t),(\bar x,\bar t))=\max\{d(x,\bar x), \sqrt{|t-\bar t|}\}$.

 For $(x, t)\in Q$ and $r>0$ sufficiently small  we consider a cylinder  $$Q_{r}(x, t)=B(x,r)\times(t-8 r^2, t) \subset B_{p}((x, t), 100 r):=\{(y,s)|d_p((x,t),(y,s))<100r\} \subset  Q$$  and
  define the upper and lower cylinders \begin{multline}\label{cylinders}
  Q^-_r(x,t)=B(x,r)\times(t-8r^2, t- 7 r^2), \quad D^-_r(x,t)=
B_\e(x,r/2)\times(t-\frac{15}{2} r^2, t_0- 7 r^2), \quad  \\ Q^+_r(x,t)=B(x,r)\times (t- r^2,t) \text{ and }
D^+_r(x,t)=B(x,r/2)\times (t- r^2, t-\frac{1}{2} r^2).
\end{multline}
We define the {\it time lag} function $T((x,t),r)=( (x,t-8r^2), r)$ and observe that for all $(x,t)\in Q$ one has
$Q^-_r(x,t)=T(Q^+_r(x,t),r)$, thus satisfying \cite[Hypothesis (1.8)]{Aimar}.
In view of  the  doubling hypothesis  \eqref{doubling} one can apply  \cite[Theorem 1.7]{Aimar} and   conclude the following version of the John-Nirenberg lemma
\begin{prop}\label{JN}
If there exists $A>0$ and $v\in L^1(Q)$, such that for all cylinders $Q_{2r}, Q_r^-,Q_r^+\subset Q$ as defined above
there is a number $C=C(Q)$ such that
\begin{equation}\label{Ipotesi H}
 \frac{1}{|Q^+_r|}\int_{Q^+_r}  \int \sqrt{(v-C(Q))^+} dx dt \le A \text{ and }
 \frac{1}{|Q^-_r|}  \int_{Q^-_r} \int\sqrt{(C(Q)-v)^+} dx dt \le A
 \end{equation}
 then there exist $C=C(\Om, C_D,A), \delta=\delta(\Om, C_D,A)>0$  such that  if we set $f=e^{-v}$ then for all $D^+_r,D^-_r\subset Q_{3r}\subset Q$ as defined above
\begin{equation}\label{john-nirenberg}
\bigg( \frac{1}{|D^+_r|} \int_{D^+_r} \! \int f^{-\delta}\bigg)\bigg( \frac{1}{| D^-_r|} \int_{D^-_r } \!\int f^{\delta}\bigg) \le C,
\end{equation}
\end{prop}

\section{Maximum principle and Harnack inequality}

Let  $\Om\subset\R^n$ be  a bounded open set and let $Q=\Om\times (0,T)$. For a function $u:Q\to \R$, and $1\le p,q$ we define the norms
\begin{equation}\label{lpq-norms}
||u||_{p,q}^q=\Big(\int_0^T (\int_\Om |u|^p dx)^{\frac{q}{p}} dt\Big)^{\frac{1}{q}},
\end{equation}
and the corresponding Lebesgue spaces $L^{p,q}(Q)=L^q([0,T], L^p(\Om))$. To simplify notations we will omit the subscripts $p,q$ when their values are clear in view of the context.  One has a useful reformulation of the Sobolev embedding theorem \eqref{Sobolev} in terms of $L^{p,q}$ spaces and a related interpolation inequality:
\begin{lemma}\label{lemma2}
(i) Let $Xu\in L^{2,2}(Q)$ and assume that for all $0<t<T$,  $u(\cdot, t)$ has  compact support in $\Om\times\{t\}$. There exists $k=k(
C_D,C_P,N)>0$ such that $||u||_{\frac{2N}{N-2}, 2} \le k ||Xu||_{2,2}.$ Here $N$ is the homogenous dimension from Proposition  \ref{longprop}(i).

(ii) If in addition we assume $u\in L^{2,\infty}(Q)$  then $u\in L^{2p',2q'}(Q)$ where $(p',q')$ have as H\"older conjugates $(p,q)$ such that $N/2p+1/q\le 1$. Moreover  there exists $k>0$ depending only on $C_D,C_P$ and $diam \ \Om$  such that
$$||u||_{2p',2q'}^2 \le k T^\theta (||w||^2_{2,\infty}+||Xw||_{2,2}^2),$$ where
$\theta=1-1/q - N/2p$.
\end{lemma}

We will consider weak solutions of \eqref{equation}, where $A_i$ and $B$ satisfy the following structure conditions: There exist constants $a,\bar a>0$ and functions
$b,c,e,f,h\in L^{p,q}(Q)$ with $p>2$, and $q$ given by  $\frac{N}{2p}+\frac{1}{q}<\frac{1}{2}$ and functions  $d,g\in L^{\al,\beta}(Q)$ with $1<\al$ and $\beta$ given by $\frac{N}{2\al}+\frac{1}{\beta}<1$ such that for a.e. $(x,t)\in Q$ and $\xi\in \R^m$ one has
\begin{eqnarray}\label{structure}
\sum_{i=1}^{{m}} A_i (x,t,u,\xi) \xi_i && \ge a|\xi|^2-b^2 u^2-f^2, \notag\\
|A(x,t,u,\xi)| && \le \bar a |\xi |+e|u|+h,\\
|B(x,t,u,\xi)| && \le c|\xi|+d|u|+g. \notag
\end{eqnarray}
In view of the conditions on $p,q,\al,\beta$ there exists  $\theta>0$ such that
\begin{eqnarray}\label{theta}
p\ge \frac{2}{1-\theta} && \text{ and } \frac{{N}}{2p}+\frac{1}{q}\le \frac{1-\theta}{2} \notag \\
\al \ge \frac{1}{1-\theta}  &&\text{ and } \frac{{N}}{2\alpha}+\frac{1}{\beta} \le 1-\theta.
\end{eqnarray}
In this note when we say that {\it a constant depends on the structure conditions \eqref{structure}}, if it depends only on \footnote{ The $||\cdot||$ norms are in the appropriate  $L^{p,q}$ or $L^{\al, \beta}$ classes}
$$a,\bar a, ||b||, ||c||,||d|| ,||e||, ||f||,||g||,||h||, N, \theta,$$ and is uniformly bounded if these quantities are so.
Here $N$ denotes  the homogenous dimension defined in Proposition \ref{longprop}(i).

 A function $u\in L^{2,\infty}(Q)$ with $Xu\in L^{2,2}(Q)$ is a {\it weak solution} of \eqref{equation}  in $\Om$ if
\begin{equation}\label{weak}
\int\int_Q -u\phi_t+X_i\phi A_i(x,t,u(x,t),Xu(x,t)) dx dt = \int\int_Q \phi(x,t) B(x,t,u(x,t),Xu(x,t)) dx dt,
\end{equation}
for every $\phi \in C^{\infty}_0(Q)$. To deal with the lack of differentiability of weak solutions along the time variable we will make use of Steklov average: For  $(x,t)\in Q$ and $h>0$ sufficiently small set
$u_h(x,t)=\frac{1}{h}\int_0^h u(x,t+s)ds$. Changing variables $t\to t+s$ and integrating in $s$ in  \eqref{weak} leads to
\begin{multline}\label{weak1}
\int\int_Q  \bigg( \p_t (u_h)\phi +X_i\phi h^{-1} \int_0^h A_i\big(x,t-s,u(x,t-s),Xu(x,t-s)\big) ds\bigg) dx dt \\ = \int\int_Q \phi(x,t) h^{-1} \int_0^h B\big(x,t-s,u(x,t-s),Xu(x,t-s)\big) dx dt,
\end{multline}
for every $\phi \in C^{\infty}_0(Q)$. Clearly the identity extends to the larger class of test functions with weak derivatives in $L^{2,2}$ and with vanishing trace.

\begin{lemma}\label{fundamental}
Let $u$ be a weak solution of \eqref{equation} in $Q=\Om\times (0,T)$ and  $\kappa$ a positive constant.
 For every $\eta\in C^{\infty}(Q)$ vanishing in a neighborhood of the parabolic boundary $\p_p Q$,  one has

\begin{itemize}
\item If $\beta\ge 1$, $\bar u=\max(0,u)+\kappa$ and $0<\tau <T$, then
\begin{multline}\label{I}
\frac{1}{\beta+1} \int_\Om \eta^2\bigg[ \bar{u}^{\beta+1}-(\beta+1)\kappa^\beta \bar u
+\beta \kappa^{\beta+1}\bigg]_{t=\tau} dx
+\frac{a\beta}{2} \int_{0}^{\tau} \int_\Om \eta^2 \bar u^{\beta+1} |X\bar u |^2 dxdt \\
 \le  \int_{0}^{\tau} \int_\Om \mathcal F\bar u^{\beta+1}dx dt +\frac{2}{\beta+1}  \int_{0}^{\tau} \int_\Om \eta |\p_t \eta| \bar u^{\beta+1} dx dt,
\end{multline}
where $\mathcal F=F\eta^2+2G\eta|X\eta|+H|X\eta|^2$, and
$$F=\beta \bigg(b^2+\frac{f^2}{\kappa^2}\bigg)+\bigg(d+\frac{g}{\kappa}\bigg)+\frac{c^2}{a}, \quad
G=e+\frac{h}{\kappa}, \text{ and }H=\frac{4\bar a^2}{a}.$$

\item  If $\beta\ge 1$, $\bar u=\max(0,u)+\kappa$ and $0<t_1<t_2<T$ then
\begin{multline}\label{II}
\frac{1}{\beta+1} \int_\Om \eta^2\bigg[ \bar u^{\beta+1}-(\beta+1)\kappa^\beta \bar u
 +\beta \kappa^{\beta+1}\bigg]_{t_1}^{t_2} dx
+\frac{a\beta}{2} \int_{t_1}^{t_2} \int_\Om \eta^2 \bar u^{\beta+1} |X\bar u |^2 dxdt \\
 \le  \int_{t_1}^{t_2} \int_\Om \mathcal F\bar u^{\beta+1}dx dt +\frac{2}{\beta+1}  \int_{t_1}^{t_2} \int_\Om \eta |\p_t \eta| \bar u^{\beta+1} dx dt,
\end{multline}
with $\mathcal F$ defined as above.

\item If $u\ge 0$ and for arbitrary $\e>0$ we let $\bar u =u+\kappa+\e$, $\beta\neq 0$ and  $0<t_1<t_2<T$ then
 \begin{multline}\label{III}
\text{sign}(\beta) \bigg(\int_\Om \eta^2 \big[\mathcal H(u)\big]_{t_1}^{t_2} dx +\frac{a\beta}{2}  \int_{t_1}^{t_2} \int_\Om
\eta^2\bar u^{\beta-1} |X\bar u|^2 dx dt \bigg) \\
\le  \int_{t_1}^{t_2} \int_\Om  \mathcal F_1 \bar u^{\beta+1} dx dt + 2  \int_{t_1}^{t_2} \int_\Om  \eta |\p_t \eta| |\mathcal H(u)| dx dt,
\end{multline}
where \begin{equation}\label{19}
\mathcal H(\bar u)=\Bigg\{ \begin{array}{ll}
\frac{1}{\beta+1} \bar u^{\beta+1} & \text{ if }\beta \neq -1,\\
\log \bar u &\text{ if }\beta=-1,
\end{array}
\end{equation}
\begin{equation}\label{F_1}
\mathcal F_1=F_1\eta^2+2G_1\eta|X\eta|+H_1|X\eta|^2,
\end{equation}
with $$F_1=|\beta| \bigg(b^2+\frac{f^2}{\kappa^2}\bigg)+\bigg(d+\frac{g}{\kappa}\bigg)+\frac{c^2}{a|\beta|}, \quad
G_1=e+\frac{h}{\kappa}, \text{ and }H_1=\frac{4\bar a^2}{a|\beta|}.$$

\end{itemize}
\end{lemma}
\begin{proof}
The argument follows closely \cite[Section 2]{MR0244638} with Steklov averages in place of the elegant convolution argument in Aronson and Serrin's original paper, as the latter seems more difficult to extend to the case of variable coefficients vector fields. Let
$$\mathcal G(u)=\bigg\{ \begin{array}{ll} \bar u^\beta-\kappa^{\beta} &\text{ for }-\infty <u\le l-\kappa, \\
 l^{\beta-1}\bar u -\kappa^\beta &\text{ for }l-\kappa\le u <\infty,\end{array}$$
where $l\ge \kappa$.
 Denote by $\mathcal H_1$  a $C^1$  function such that $\mathcal H_1'(s)=\mathcal G(s)$ for all $s\in \R$. In particular we can assume that $\mathcal H_1$ coincides with the integrand function of the lhs of \eqref{I} for $s < l-\kappa$.
 Choose $0<\tau<T$ as in \eqref{I} and denote by $\chi_{(0,\tau)}$ the characteristic function of the interval $(0,\tau)$. Let  $\eta\in C^{\infty}(Q)$ vanishing in a neighborhood of $\p_p Q$ and set
$\phi(x,t)=\eta^2(x,t) \mathcal G(u_h(x,t) )\chi_{(0,\tau)}(t)$ in \eqref{weak1},  to obtain
\begin{multline} \int_\Om \eta^2 \bigg[\mathcal H_1(u_h) dx \bigg]_{t=\tau} - \int_0^\tau \int_\Om 2\eta \p_t \eta \mathcal H_1(u_h) dx dt +
\\   \int_0^\tau \int_\Om X_i\phi h^{-1} \int_0^h A_i\bigg(x,t-s,u(x,t-s),Xu(x,t-s)\bigg) ds  dx dt \\ = \int\int_Q \phi(x,t) h^{-1} \int_0^h B\bigg(x,t-s,u(x,t-s),Xu(x,t-s)\bigg) dx dt.
\end{multline}
At this point we can let $h\to 0$. In the LHS we let $l \to \infty$  so that $\mathcal H_1$ reduces to the integrand
function of (3.6). In the right hand side we  make use of the structure conditions to conclude \eqref{I} in the same fashion as in \cite{MR0244638}. The estimates \eqref{II} and \eqref{III} are proved in the same way.
\end{proof}

\subsection{Proof of the maximum principle.}
In this section we prove the maximum principle Theorem \ref{T:max} . We also note that using a similar argument one obtains a corresponding minimum principle for weak solutions.

\begin{dfn}\label{norma3barre}
 Let $|||w|||=\sup_{p,q} ||w||_{p,q}$ where the sup is taken over all the pairs $(p,q)$ satisfying
$n/2p+1/q\le 1-\theta$ and $p\ge 1/(1-\theta)$, where $\theta$ has been defined in  \eqref{theta}.
\end{dfn}

The proof of Theorem \ref{T:max} is accomplished in two steps as described in Lemma \ref{step1} and Lemma \ref{step2}. We remark that without loss of generality one can assume  $M<0$. The general case can then be easily derived by substituting $u$ with $u-M-\delta$ for an arbitrary $\delta>0$ and invoking the structure conditions \eqref{structure}.

\begin{lemma}\label{step1} Let $M$ be the constant defined in Theorem \ref{T:max}. If  $M<0$ then  there exists a positive constant $C_1$ depending\footnote{The same dependance holds for the constants $C_2,C_3,...$ used in the proof} only on $C_D,C_L,C_P$ and the structure conditions \eqref{structure} such that the function
$\tilde u=\max(0,u)$ satisfies
\begin{equation}\label{20}
u(x,t)\le C_1 (||\tilde u||_{2,\infty}+||X\tilde u||_{2,2}+\kappa),
\end{equation}
for a.e.  $(x,t)\in Q$ and with $\kappa=||f||+||g||$.
\end{lemma}
\begin{proof} As in proof of Lemma \ref{fundamental} we set $\bar u=\tilde u+\kappa$ and note that since $\mathcal G(\tilde u)$ vanishes in a neighborhood of the parabolic boundary of $Q$ then \eqref{I} holds with no need of the cut-off function $\eta$,
\begin{multline}\label{21}
\frac{1}{\beta+1}\int_\Om \bigg[ \bar{u}^{\beta+1}-(\beta+1)k^\beta \bar u
+\beta k^{\beta+1}\bigg]_{t=\tau} dx
+\frac{a\beta}{2} \int_{0}^{\tau} \int_\Om \bar u^{\beta+1} |X\bar u |^2 dxdt \\
 \le  \int_{0}^{\tau} \int_\Om  F\bar u^{\beta+1}dx dt \end{multline}
where $0<\tau<T$ and
$F=\beta \bigg(b^2+\frac{f^2}{\kappa^2}\bigg)+\bigg(d+\frac{g}{\kappa}\bigg)+\frac{c^2}{a}.$
Set $\beta_1=(\beta+1)/2$ and $v=\bar u^{\beta_1}$. Letting $\tau\to T$ (the height of the parabolic cylinder)  in \eqref{21} one obtains $a||Xv||^2_{2,2}\le 2\frac{\beta_1^2}{\beta} ||Fv^2||_{1,1}.$ In view of the structure conditions \eqref{structure} one has that there exists $C_2>0$  such that $\beta^{-1}||Fv^2||_{1,1}\le C_2 |||v|||^2,$ and consequently
\begin{equation}\label{24}
a||Xv||^2_{2,2}\le C_2 \beta_1^2|||v|||^2.
\end{equation}
Young inequality and \eqref{21} yield $||v||^2_{2,\infty}\le C_3 \beta_1^2|||v|||^2$ for some $C_3>0$.
The latter, Lemma \ref{lemma2} and \eqref{24} imply
\begin{equation}\label{26}
|||v|||^\sigma\le 2(k+1) ( ||v||^2_{2,\infty} +||Xv||_{2,2}^2) \le C_4 \beta_1^2 |||v|||^2,
\end{equation}
with $\sigma=1+2\theta/N$. A standard iteration argument\footnote{See \cite[page 95]{MR0244638} for details} in the exponent of $v$ leads to \eqref{20}.
\end{proof}

The second step concludes the proof of the maximum principle.

\begin{lemma}\label{step2}
 If the constant $M$ defined in Theorem \ref{T:max} is negative, then  there exists a positive constant $C$ depending\ only on $C_D,C_L,C_P$ and the structure conditions \eqref{structure} such that the function
$\tilde u=\max(0,u)$ satisfies
\begin{equation}\label{32}
 ||\tilde u||^2_{2,\infty} \le C\kappa^2 \text{ and } ||X\tilde u||^2_{2,2} \le C\kappa^2,
\end{equation}
 with $\kappa=||f||+||g||$.
\end{lemma}

\begin{proof}
Let $\beta=1$ in \eqref{II}, with the cut-off function $\eta$ omitted, and apply H\"older inequality and the interpolation in Lemma \ref{lemma2}(ii) to obtain
\begin{multline}\label{28}
\frac{1}{2} \int_\Om \bigg[\tilde u^2\bigg]_{t_1}^{t_2} dx + \frac{a}{2}\int_{t_1}^{t_2} \int_\Om |X\tilde u|^2 dx dt
\\
\le \int_{t_1}^{t_2} \int_\Om F\bar u^2 dx dt
\\ \le 3(k+1) C(t_2-t_1)^\theta (||\tilde u||^2_{2,\infty}+||X\tilde u||^2_{2,2}+\kappa^2),
\end{multline}
with $\theta$ as in \eqref{theta} and for a constant $C$ as in the statement of this lemma.
Next, we let {$3(k+1) C\mu^\theta <\min(a,1)/4$}, choose $t \in ]t_1, {t_1+\mu}[$ and set $Y(t)=\int_{\Om} \tilde u^2(x,t)dx$. Estimate \eqref{28} yields
\begin{equation}\label{29}
Y(t)+{\frac{a}{4} }\int_{t_1}^t \int_\Om |X\tilde u|^2 dx dt \le {\frac{\min{(a,1)}}{2}} (||\tilde u||_{2,\infty}+\kappa^2) + Y(t_1).
\end{equation}
This implies  $Y(t_1+\mu) \leq \kappa^2/2 + Y(t_1)$. Iterating this inequality from one time interval to the next yields $Y(t) \le 2 ^{1+t\mu^{-1}}\kappa^2$ which for $t=T$
gives the first inequality in \eqref{32}. The second inequality follows immediately from the latter and from \eqref{28}.
\end{proof}

%
%

\subsection{Proof of the Harnack inequality} In this section we prove Theorem \ref{T:harnack}.
 Let $Q_r,Q^{\pm}_r,D^{\pm}_r$ denote the cylinders defined in \eqref{cylinders}.

\begin{lemma}\label{supinf}
Let $u\ge 0$ be a weak solution of \eqref{equation} in $Q=\Om\times (0,T)$, $\theta$ be as in \eqref{theta}, $N$ the homogenous dimension (Proposition \ref{longprop}(i)) and $0<r<20R$ with $R$ as in \eqref{doubling}. There exist a constant  $C>0$ depending on $diam(\Om)$,$C_D,C_P,C_L$ and on the structure constants \eqref{structure} such that for every $Q_{2r}(x_0,t_0)\subset Q$ and every exponent of the form $\beta_0=(1+2\frac{\theta}{N})^{-h}(2+2\frac{\theta}{N})^{-1}$, $h=1,2,...$
one has
\begin{equation}\label{sup}
\text{ esssup}_{  D^-_{\frac{r}{2}} }
|\bar u^{\beta_0}|
\le
C  |||\bar u^{\beta_0}||_{Q^-_{\frac{r}{2}}}
 \end{equation}
 and
 \begin{equation}\label{inf}
\big( \text{ essinf}_{D^+_{\frac{r}{2}}} |\bar u^{\beta_0}|\big)^{-1} \le C
 |||\bar u^{-\beta_0}|||_{Q^+_{\frac{r}{2}}}
 \end{equation}
where the norm $|||\cdot|||$ is defined in Definition \ref{norma3barre}, we have let $\bar u= u+k+\e$ with $\e>0$ arbitrary and $k$ as in the statement of Theorem \ref{T:harnack}.
\end{lemma}
\begin{proof}
Since translations in the time variable do not affect the structure constants in \eqref{structure} we will prove the two estimates
\eqref{sup} and \eqref{inf} separately in cylinders
 $$S(s)=B(x_0,sr)\times  \bigg(\frac{1-s}{6}r^2, \frac{1+s}{6}r^2\bigg),$$
for each $1/3\le s \le 1/2$
Let $\beta\neq -1$. Substitute in \eqref{III} $v=\bar u^{\beta_1}$ with $\beta_1=(\beta+1)/2$ to obtain
\begin{multline}\label{42}
\text{sign}(\beta) \bigg(\int_\Om [\eta^2 v^2]_{t_1}^{t_2} dx +\frac{a\beta}{2{\beta_1}^2}  \int_{t_1}^{t_2} \int_\Om
\eta^2 |X v|^2 dx dt \bigg) \\
\le  \int_{t_1}^{t_2} \int_\Om ( \mathcal F_1+2 |\beta+1|^{-1} \eta |\p_t \eta| )v^2 dx dt.
\end{multline}
To prove \eqref{sup} we let  $\beta>-1$. For $1/2\le l'<l \le1/3$  choose $\eta\in Lip_0(S(l))$ with $\eta=1$ on $S(l')$
and $|X\eta|\le \frac{C}{(l-l')r}$ and $|\p_t \eta|\le \frac{C}{(l-l')r^2}$. Substituting $\eta$ in \eqref{42} and using H\"older inequality we obtain the following estimate for the RHS,
\begin{multline}\label{43}
||( \mathcal F_1+2 |\beta+1|^{-1} \eta |\p_t \eta| )v^2 ||_{1,1}
\le \max(|\beta|,|\beta+1|^{-1},|\beta|^{-1}) C(l-l')^{-2} r^{-2} |||v|||^2_{S(l)},
\end{multline}
where here and for the rest of the proof we indicate by $C$   constants as in the statement of the lemma.
In view of \eqref{43} we obtain for all $\beta>-1$, $\beta\neq 0$
\begin{equation}
||\eta v^2||_{2,\infty}^2+a||\eta v^2||_{2,2} \le \frac{C(1+|\beta|^{-2}+|\beta|^{-1})\big(1+(\frac{\beta+1}{2})^2\big)}{r^2(l-l')^2} |||v|||^2_{S(l)}.
\end{equation}
By means of the Sobolev inequality (Lemma \ref{lemma2}) the latter yields the basic iteration formula
\begin{equation}\label{44}
|||(\eta v)^\sigma|||^{2/\sigma}\le  \frac{C(1+|\beta|^{-2}+|\beta|^{-1})\big(1+(\frac{\beta+1}{2})^2\big)}{r^2(l-l')^2} |||v|||^2_{S(l)},
\end{equation}
where the {\it gain} in integrability is $\sigma=1+2\frac{\theta}{N}$, and $N$ is the homogenous dimension. Set the iteration step to be $\beta_m=\beta_0\sigma^m$, for $m=1,2,...$. To avoid the exponent $\beta=0$, it is convenient to set $\beta_0^{-1}=\sigma^h(1+\sigma)$ for any $h=1,2,...$.  To carry out the iteration we let $l_m=\frac{1}{3}+\frac{2^{-m-1}}{3}$
and $l_m'=\frac{1}{3}+\frac{2^{-m-2}}{3}$ and proceed as in Moser's original paper arriving eventually at the estimate
$$\text{ esssup}_{  S(\frac{1}{3})}
|\bar u^{\beta_0}|
\le
 C|||\bar u^{\beta_0}|||_{S(\frac{1}{2})}.$$
  The proof of \eqref{inf} is very similar and is omitted.
\end{proof}
%
%

%
%

%
%
Next, we turn our attention to the 'bridge' step in the Moser iteration scheme, where the sup and the inf estimates
established above are linked by means of the John-Nirenberg lemma.

\begin{lemma}\label{bridge}
Let $u\ge 0$ be a weak solution of \eqref{equation} in $Q=\Om\times (0,T)$ and $0<r<20R$ with $R$ as in \eqref{doubling}. There exists $B>0$ depending on $diam(\Om)$,$C_D,C_P,C_L$ and on the structure constants \eqref{structure} such that for every $Q_{2r}(x_0,t_0),Q^+_r(x_0,t_0),Q^-_r(x_0,t_0)\subset Q$
one has
\begin{equation}\label{bmo}
\frac{1}{|Q^+_r|}\frac{1}{|Q^-_r|} \int_{Q^+_r}\int_{Q^-_r} \sqrt{\bigg(\log \bar u(y,s) - \log \bar u(x,t)\bigg)^+} dyds \ dxdt \le B,\end{equation}
where we have let $\bar u= u+k+\e$ with $\e>0$ arbitrary and $k$ as in the statement of Theorem \ref{T:harnack}.
\end{lemma}
\begin{proof}
Set $\beta=-1$  and $v=\log \bar u$ in \eqref{III} to obtain
\begin{multline}\label{52}
-\int_\Om \eta^2 (v(\cdot,t_2)-v(\cdot,t_1)) dx +\frac{a}{2} \int_{t_1}^{t_2}\int_\Om \eta^2|Xv|^2 dx dt \\
\le  \int_{t_1}^{t_2}\int_\Om  \mathcal F_1  dx dt +2  \int_{t_1}^{t_2} \int_\Om  \eta |\p_t \eta| |v| dx dt,
\end{multline}
where $F_1$ is as in \eqref{III}. Next we choose the cut-off function so as to match the requirements of the weighted Poincar\'e inequality \eqref{weighted}: $\eta(x,t)=\xi(x)w(t)$ with $\xi(x)=\sqrt{\phi}(d(\cdot, x_0))$ and where $\phi$ satisfies the hypothesis in Proposition \ref{longprop}(iii). In this way $\xi\in Lip_0(B(x_0,r))$ and can be chosen identically equal to one in $B(x_0,r/2)$. We also let $w=1$ if $t>t_1$ and $w=0$ for $t<t_1/2$. Setting
$$V(t)=\frac{\int_\Om \xi^2 v(x,t)dx}{\int_\Om \xi^2dx},$$
in \eqref{52}  and applying \eqref{weighted} one obtains
\begin{multline}\label{page107}
V(t_2)-V(t_1)+ \frac{a}{2} \frac{C\kappa^{2/n}}{\bigg(\int_\Om \xi^2 dx\bigg)^{-1-\frac{2}{N}}}  \int_{t_1}^{t_2} \int_\Om |v-V(t)|^2 \xi^2 dx dt
\\
\le \frac{1}{\int_\Om \xi^2dx} \int_{t_1}^{t_2} \int_\Om \mathcal F_1 dx dt,
\end{multline}
where $\kappa=\inf_{x\in \Om} |B(x,r)| r^{-\log_2 C_D}>0$. Since $|B(x_0,r)|^{-2/N} \ge C(diam \Om)>0$  the estimate \eqref{page107} yields
\begin{equation}\label{53}
\frac{dV}{dt}+  \frac{C_1}{|B(x,r)|} \int_{B(x_0,r)} (v-V)^2dx \le \frac{C_2}{|B(x,r)|}\int_{B(x,r)} \mathcal F_1 dx,
\end{equation}
for every $0<t<T$ and for constants $C_1,C_2>0$ depending on the structure conditions and on $diam(\Om), C_D,C_L$ and $C_p$. The conclusion now follows from \cite[Lemma 7]{MR0244638} as in the Euclidean setting.
\end{proof}
The previous lemma tells us that $\log \bar u$ is in the parabolic BMO space. Applying Proposition \ref{JN} we conclude
\begin{cor}\label{bridge-1}
Let $u\ge 0$ be a weak solution of \eqref{equation} in $Q=\Om\times (0,T)$ and $0<r<20R$ with $R$ as in \eqref{doubling}. There exist  $C, \delta>0$ depending on $diam(\Om)$,$C_D,C_P,C_L$ and on the structure constants \eqref{structure} such that for every $Q_{2r}(x_0,t_0),Q^+_r(x_0,t_0),Q^-_r(x_0,t_0)\subset Q$
one has
\begin{equation}\label{bmo-1}
\bigg( \frac{1}{|D^+|} \int_{D^+_r} \! \int \bar u^{-\delta} dx dt \bigg)\bigg( \frac{1}{| D^-|} \int_{D^-_r}\! \int \bar u^{\delta} dx dt \bigg) \le C,
\end{equation}
where we have let $\bar u= u+k+\e$ with $\e>0$ arbitrary and $k$ as in the statement of Theorem \ref{T:harnack}.
\end{cor}
In order to complete the proof of the Harnack inequality in Theorem \ref{T:harnack} one needs to link \eqref{bmo-1} with the RHS of inequalities
\eqref{sup} and \eqref{inf}.  This can be  easily accomplished following the same argument at the end of page 106 in \cite{MR0244638}.

\section{Stability of the homogenous structure and of the Poincar\'e
inequality in the Riemannian approximation to a Carnot-Caratheodory
space}

Let $X=(X_1,...,X_m)$ be smooth vector fields in $\R^n$ such that, together with all their commutators up to step $r$, they generate  $\R^n$ at every point. Following \cite[page 104]{NSW} we define the collections of commutators of same degree
$$X^{(1)}=\{ X_1,...,X_m\}, \ X^{(2)}=\{ [X_1,X_2],...,[X_{m-1},X_m]\}, \ etc. ...$$
Indicate by $Y_1,...,Y_p$ an enumeration of the components of $X^{(1)}, X^{(2)},...,X^{(r)}$ such that $Y_i=X_i$ for every $i\leq m$.  If $Y_k\in X^{(i)}$ we set the degree of $Y_k$ to be $d(Y_k)=d(k) =i$.

{For $\bar \e\le 1$ and for each $\e\in (0,\bar \e)$}  consider rescaled vector fields, and a suitable subfamily of their commutators
$$X_i^\e=\Bigg\{ \begin {array}{ll} Y_i & \text{ if } i\leq m,  \\
\e^{d(i) -1 } Y_i &\text{ if } m+1 \leq i \leq p
\end{array},\quad Y_i^\e=\Bigg\{ \begin {array}{ll} X^\e_i &\text{ if } i\leq p,  \\
 Y_{i-p+m} &\text{ if } p+1 \leq i \leq 2p -m
\end{array}$$
We will also extend the degree function, setting $d_\e(i)=1$ for all $i\leq p$, and $d_\e(i) = d(i-p+m)$ if  $i\geq p+1$.  In order to simplify notations we will denote $X=
 X^0$, $Y=Y^0$, $d_0=d$
 and use the same notation for both families of vector fields
 (dependent or independent of $\e$).

Note that for every $\e$ the sets $\{Y_i^\e\}$ extends the family $(X_i^\e)$ to a  new family of vector fields satisfying assumption (I) on page 107 \cite{NSW}: There exist smooth functions $c_{jk}^l$, depending on $\e$, such that
$$[Y^\e_j,Y^\e_k] =\sum_{d_\e(l)\leq d_\e(j) + d_\e(k)} c_{jk}^l {Y_l^\e}$$
and $$\{Y^\e_j\}_{j=1}^{2p-m} \text{ span } \R^n  \text{ at every point }.$$

\begin{rmrk}
Note that the coefficients $c_{jk}^l$ will be unbounded as $\e\to 0$. In principle this could be a problem  as the
doubling constant in the proof in \cite{NSW} depends indirectly from the $C^r$ norm of these functions. \end{rmrk}

Next we consider the Carnot-Caratheodory metric $d_\e(\cdot, \cdot)$  associated to the family of vector fields $(X_1^\e, ...,X_{p}^\e).$
%
Note that for $1\ge \bar \e \ge \e>0$,   $d_\e$ corresponds to the distance function of a Riemannian metric $g_\e$, which degenerates to a sub-Riemannian metric as $\e\to 0$.

Following \cite[page 110]{NSW}, for every $n-$tuple $I=(i_1,...,i_n)$, we define the coefficient $$\lambda^\e_I(x)=\det (Y^\e_{i_1}(x),...,Y^\e_{i_n}(x)),$$
for  $\e\geq 0$. For a fixed constant $0<C_{2,\e}<1$, choose $I_\e=(i_{\e 1},...,i_{\e n})$ such that
\begin{equation}\label{bestI}|\lambda^\e_{I_\e}(x)|r^{d_\e(I_\e)} \ge C_{2,\e} max_J |\lambda^\e_J(x)|r^{d_\e(J)}.\end{equation}
and denote $J_\e$ the family of remaining indices, so that
$\{Y_{i_{\e j}}: i_{\e j} \in I_\e\} \cup \{Y_{i_{\e k}}: i_{\e k} \in J_\e\}$ is the complete list $Y^\e$.
We will refer to $I_0$ { as the choice corresponding to the $n-$tuple $Y_{i,1},...,Y_{i,n}$ realizing
\eqref{bestI} exactly as in the   setting of \cite{NSW}.}
The core of the result of Nagel, Stein and Wainger \cite[Theorem 1]{NSW}, is to prove that
if  $v$ and $x$ are fixed, and
$$Q_\e( r)=\{u\in \R^n: |u_j| \leq  r^{d_\e( i_{\e j})}\}$$
is a weighted cube in $\R^n$, then
 $|\lambda^\e_{I_\e}(x)|$ provides an
estimates of the Jacobian of the exponential mapping
$\to \Phi_{\e, v, x}(u)$ defined for   $u\in Q(r)$ as \begin{equation}\label{phi}
\Phi_{\e, v, x}(u) = exp\Big(\sum_{i_{\e j}\in I_\e} u_j Y^\e_{i_{\e j}} + \sum_{i_{\e k}\in J_\e} v_k Y^\e_{i_{\e k}} \Big)(x).\end{equation}

\bigskip

More precisely, for $\e\geq 0$ and fixed \cite[Theorem 7]{NSW} states
\begin{thrm}\label{MAINNSWeps} For every $\e\geq 0$,  and  $K\subset \subset \R^n$ there exist $R_\e>0$ and constants $0<C_{1, \e}, C_{2, \e} <1$  such that for every $x\in K$ and  $0<r<R_\e$, if $I_\e$ is such that \eqref{bestI}
holds, then

\begin{itemize}
\item[i)] if $|v_k| \leq C_{2 \e}r^{d(i_{\e k})}$,
$\Phi_{\e, v, x } $ is one to one on the box $Q_\e(C_{1, \e} r)$
\item[ii)] if $|v_k| \leq C_{2 \e}r^{d(i_{\e k})}$
the Jacobian matrix of $\Phi_{\e, v, x}$ satisfies on the cube $Q_\e(C_{1, \e} r)$
$$\frac{1}{4} |\lambda^\e_{I_\e} (x)| \leq |J\Phi_{\e, v, x}| \leq 4 |\lambda^\e_{I_\e} (x)| $$
\item[iii)]
$$\Phi_{\e, v, x}(Q_\e (C_{1, \e} r))  \subset B_{\e}(x,   r) \subset \Phi_{\e, v, x}(Q_\e(C_{1, \e} r/C_{2, \e})) $$
\end{itemize}
\end{thrm}

\bigskip

A direct consequence of this fact is that the measure of the Ball centered in $x$ can be estimated by the measure of the cube and the Jacobian determinant of $\Phi_{\e,v,x}$.  Varying the central point $x$, \cite{NSW} obtain:

\begin{thrm}\label{balmeasure}(\cite[Theorem 1]{NSW})  For every $\e\geq 0$,  and  $K\subset \subset \R^n$ there exists a constant $R_\e>0$ and constants $C_{3 \e}, C_{4\e}>0$ such that $x\in K$ and  $0<r<R_\e$
\begin{equation}\label{nsw}
C_{3 \e}\sum_I |\lambda^\e_I(x)| r^{d(I)} \le |B_\e(x,r)|\le C_{4 \e} \sum_I |\lambda^\e_I(x)| r^{d(I)},
\end{equation}
\end{thrm}

To prove that  that these inequalities hold uniformly in $\e$, it is enough to study the constants $C_{1\e}$ $C_{2\e}$ and show that they do not vanish as $\e\to 0$. Without loss of generality we can assume that both constants are non-decreasing in $\e$,
otherwise we consider a new pair of constants $\tilde C_{i,\e}=\inf_{s\in [\e,\bar \e]} C_{i,s}$, for $i=1,2$.

\begin{prop}\label{indipendenza}
For every $\e\in[0,\bar \e]$, the constants $R_\e, C_{1, \e}$ and $C_{2, \e}$ in Theorem \ref{MAINNSWeps} may be chosen to be  independent of $\e$, depending only on  the $C^{r+1}$ norm of the vector fields, $\bar \e$, and on $K$  .
\end{prop}

\begin{proof} The proof is split in two cases: First we study the range $\e<r<R_0$ which  roughly corresponds
to the balls of radius $r$ having a sub-Riemannian shape. In this range we show that one can select
the constants $C_{i,\e}$ to be approximately $C_{i,0}$. The second case consists in the analysis of the range
$r<\e<\bar \e$. In this regime the balls are roughly of  Euclidean shape and we show that the constants
$C_{i,\e}$ can  be approximately chosen to be $C_{i,\bar \e}$.

 Let us fix $\e>0$, $R=R_0$ and $r<R_0$.
We can start by describing the family $I_\e$ defined in  (\ref{bestI}), which maximize $\lambda^\e_I(x)$.
We first note that for every $\e>0$ and for every $i$, $m+1 \leq i \leq p$
we have
  \begin{equation}\label{Yer}
 Y^\e_{i}r^{d_\e(i)} = \e^{d(i)-1} r Y_i, \quad Y^\e_{i+p-m}r^{d_\e(i+p-m)} =  r^d(i) Y_i.
  \end{equation}

Case 1: For every  $\e<r<R_0$ the indices $I_\e$ defined by the maximality condition (\ref{bestI}) coincide with indices of the family $I_0$ and do not depend on $\e$.
On the other hand
$$
\{Y_k: k\in J_\e\} = \{Y_{i_{0,k}}: i_{0,k}\in J_0\} \cup\{\e^{d(i_{0,k})-1} Y_{i_{0,k}}:  i_{0,k}\in I_0, i_{0,k} \geq m+1 \}  $$ $$  \cup\{\e^{d(i_{0,k})-1} Y_{i_{0,k}}:  i_{0,k}\in J_0 {, i_{0,k}>m} \}.$$
In correspondence with this decomposition of the set of indices we define a splitting in the
$v-$variables in \eqref{phi} as $$v=(\hat v, {\tilde v, \bar v}).$$
Consequently for every $\e<r$ the function  $\Phi_{\e, v, x}(u)$  reduces to
\begin{equation}\label{e0}
\Phi_{\e, v, x}(u)   =
exp\Big(\sum_{i_{\e j}\in I_\e} u_j Y^\e_{i_{\e j}} + \sum_{i_{\e k}\in J_\e} v_k Y^\e_{i_{\e k}} \Big)(x)= exp\Big(\sum_{i_{0 j}\in I_0} u_j Y^0_{i_{0 j}} + \sum_{i_{\e k}\in J_\e} v_k Y^\e_{i_{\e k}} \Big)(x)=\end{equation}
$$exp\Big(\sum_{i_{0 j}\in I_0} u_j Y^0_{i_{0 j}} + \sum_{i_{0 k}\in J_0} \hat v_k Y^0_{i_{0 k}} +
\sum_{i_{0 k}\in I_0, i >m } {\tilde v}_k \e^{d(i_{0 k})-1}  Y^0_{i_{0 k}}
+ \sum_{i_{0 k}\in J_0  {, i_{0,k}>m}  }{\bar v}_k \e^{d(i_{0 k})-1} Y^0_{i_{0 k}}
\Big)(x)=$$
$$=\Phi_{0, \hat v_k + \bar v_k \e^{d(i_{0 k})-1}, x}(u_1, \cdots u_m, u_{m+1} + {\tilde v}_{m+1} \e^{d(i_{0 m+1})-1}, \cdots, u_{n} + {\tilde v}_{n} \e^{d(i_{0 n})-1}).$$

Let us define mappings
$$F_{1,\e,v}(u)=\bigg(u_1,...,u_m, u_{m+1}+\tilde v_{m+1}\e^{d(i_{0  m+1})-1},...,u_{n} + {\tilde v}_{n} \e^{d(i_{0 n})-1}\bigg),$$
and
$$F_{2,\e}(v)=\bigg(\hat v_1 + \bar v_1 \e^{d( i_{01})-1 }, ..., \hat v_{2p-m} + \bar v_{2p-m} \e^{d( i_{0, 2p-m})-1}\bigg).$$
In view of \eqref{e0} we can write \begin{equation}\label{compos}
\Phi_{\e, v, x}(u) =\Phi_{0, F_{2,\e}(v), x}(F_{1,\e,v}(u)).
\end{equation}

Note that for any $\e\ge 0$ and for a fixed $v$, the mapping $u\to F_{1,\e,v}(u)$ is invertible and volume preserving in all $\R^n$. Moreover $J\Phi_{\e, v, x}(u) =J\Phi_{0, F_{2,\e}(v), x}(F_{1,\e,v}(u)).$
In view of \eqref{compos} and of Theorem \ref{MAINNSWeps}, as a function of $u$, the mapping $\Phi_{\e,v,x}(u)$ is defined,  invertible, and satisfies the Jacobian estimates in Theorem  \ref{MAINNSWeps} (ii)
$$\frac{1}{4} |\lambda^0_{I_0} (x)| \leq |J\Phi_{0, F_{2,\e}(v),x}(F_{1,\e,v}(u))|=|J\Phi_{\e, v,x}(u)| \leq 4 |\lambda^0_{I_0} (x)| $$ for all $u$ such that  $F_{1,\e,v}(u)\in Q_0(C_{1,0}r)$  and for $v$ such that
$$|F_{2,\e}^k(v)|=|\hat v_k + \bar v_k \e^{d(i_{0 k})-1}|\leq C_{2, 0}r^{d(i_{0 k})}, $$$$ |u_1|\leq  C_{1, 0}r^{d(i_{0 1})}   \cdots |u_m |\leq C_{1, 0} r^{d(i_{0 m})}, |u_{m+1} + {\tilde v}_{m+1} \e^{d(i_{0 m+1})-1}|\leq C_{1, 0} r^{d(i_{0 m+1})},$$
when $k=1,...,2p-m$.

The completion of the proof of Case 1 rests on the following two claims:

{\bf Claim 1} let  $\e<r<R_0$. There exists $C_6>0$, independent of $\e$, such that for all $v$ satisfying  $|v_k| \leq C_6r^{d(i_{\e k})}$ one has  $|F_{2,\e}^k(v)|=|\hat v_k + \bar v_k \e^{d(i_{0 k})-1}|\leq C_{2, 0}r^{d(i_{0 k})} .$

{\bf Proof of the claim:}
If we choose  $C_6< \min\{C_{1, 0}, C_{2, 0}\}$ and
$$|\hat  v_k|, |\tilde v_k|, |\bar v_k|\leq \min\{C_{1, 0}, C_{2, 0}\}\frac{r^ {d(i_{\e k})}}{4}, \quad |u_j|\leq C_{1, 0}\frac{r^{d(i_{\e j})}}{4},$$
it follows that
$$|\hat v_k|\leq C_{2, 0}\frac{r^ {d(i_{0 k})}}{4}, \quad   |\tilde v_k|, |\bar v_k|  \leq C_{1, 0}\frac{r}{4},\quad |u_j| \leq C_{1, 0}\frac{r^{d(i_{\e j})}}{4}.$$
So that
$$|\hat v_k|\leq C_{2, 0}\frac{r^ {d(i_{0 k})}}{4}, \quad \e^{d(i_{0 k})-1} |\tilde v_k|,  \quad \e^{d(i_{0 k})-1}|\bar v_k|  \leq C_{1, 0}\frac{r^{d(i_{0 k})} }{4},\quad |u_j| \leq C_{1, 0}\frac{r^{d(i_{0 j})}}{4},$$ completing the proof of the claim.

{\bf Claim 2} Let  $\e<r<R_0$ and $v$ fixed such that $|v_k| \leq C_6 r^{d(i_{\e k})}$ for $k=1,...,2p-m$. One has that $$ Q_\e(C_5^{-1} r)\subset F_{1,\e,v}^{-1}( Q_0(C_{1,0} r)) \subset  Q_\e(C_5 r)$$ for some constant $C_5>0$ independent of $\e\ge 0$.

{\bf Proof of the claim:} Choose $C_5$ sufficiently large so that $2\max \{ C_5^{-1}, C_6 \}\le C_{1,0}$ and observe that
if $u\in Q_\e(C_5^{-1} r)$ then for $k=1,...,m$ we have $|u_k|\le C_{1,0}r^{d(i_{\e,k})}=C_{1,0}r^{d(i_{0,k})}$
while for $k=m+1,...,n$ we have $|F^k_{1,\e,v}(u)|=|u_k+\tilde v_k \e^{d(i_{0k})-1}|\le
\max\{C_5^{-1}, C_6 \} r^{d(i_{0 k})} (1+\bar \e^{d(i_{0 k})-1})\le C_{1,0} r^{d(i_{0 k})}$. This proves the first inclusion in the claim. To establish the second inclusion we choose $C_5$ large enough so that $2(C_{1,0}+C_{2,\bar \e})\le C_5$ and observe that if $F_{1,\e,v}(u)\in  Q_0(C_{1,0} r)$ then
for $k=m+1,...,n$ one has $|u_k|\le |u_k+\tilde v_k \e^{d(i_{0k})-1}|+|\tilde v_k| \e^{d(i_{0k})-1}\le
2(C_{1,0}+C_{2,\bar \e}) r^{d(i_{0 k})}\le C_5 r^{d(i_{0 k})}$. The corresponding estimate for the range
$k=1,...,m$ is immediate.

In view of Claims 1 and 2, and of  Theorem \ref{MAINNSWeps}
It follows that for $\e<r$ and these choice of constants (independent of $\e$)\footnote{$R_0$ in place of $R_\e$,  $C_5$ in place of $C_{1,\e}$ and
$C_6$ in place of $C_{2,\e}$} the function $\Phi_{\e, v,x}(u) $ is invertible
on $Q_0(C_{1,0}r)$ and  i), ii) and iii) are satisfied.

\bigskip

Case 2: If  $r$ is  such that $r<\e<\bar \e$ we can apply \ref{Yer} and deduce that,
if $$(Y^0)_{ {i_{0,1}} \in I_0} =
\{Y_{i_{0,1}}, \cdots, Y_{i_{0,n}}\}$$
then the maximality condition  (\ref{bestI}) selects the following family of vector fields:
 $$(Y^\e)_{{i_{\e,1}} \in I_\e}=\{\e^{d(i_{0,1})-1} Y_{i_{0,1}}, \cdots, \e^{d(i_{0,n})-1} Y_{i_{0,n}} \}$$

The complementary family $J_\e$ becomes $$
\{Y^\e_{i_{\e k}}: i_{\e k} \in J_\e\} = \{\e^{d(i_{0,k})-1} Y_{i_{0,k}}: i_{0,k}\in J_0\} \cup \{ Y_{i_{0,k}}: i_{0,k}\in J_0\}\cup
\{Y^0_{i_{0,k}}: i_{0,k}\in I_0,  i_{0,k} \geq m+1\}$$
If we denote $A_\e$, $B_\e$ and $C_\e$ these three sets, and split the $v-$variable as $v=(\hat v, \tilde v,\hat v)$, then
it is clear that
$$Y\in A_\e \text{ iff }  \frac{{\bar \e}^{d(i_{0,k})-1} }{\e^{d(i_{0,k})-1} }Y\in A_{\bar \e}, $$ and in this case the values of $d_\e$ and $d_{\bar\e}$ are the same on the corresponding indices. Analogously $$ Y\in B_\e \text{ iff } Y\in B_{\bar \e}, \quad Y\in C_\e \text{ iff }  Y\in C_{\bar \e},$$
and the degrees are the same.

Consequently in this case, for every $\e>r$ the function  $\Phi_{\e, v, x}(u)$  reduces to
$$\Phi_{\e, v, x}(u)   =
exp\Big(\sum_{i_{\e j}\in I_\e} u_j Y^\e_{i_{\e j}} + \sum_{i_{\e k}\in J_\e} v_k Y^\e_{i_{\e k}} \Big)(x)= exp\Big(\sum_{i_{0 j}\in I_0} u_j \e^{d(i_{0,k})-1} Y^0_{i_{0 j}} + \sum_{i_{\e k}\in J_\e} v_k Y^\e_{i_{\e k}} \Big)(x)=$$
$$exp\Big(\sum_{i_{0 j}\in I_0} u_j \frac{{ \e}^{d(i_{0,k})-1} }{\bar \e^{d(i_{0,k})-1} }Y^{\bar \e}_{i_{0 j}} + \sum_{i_{0 k}\in J_0} \hat v_k Y^0_{i_{0 k}} +
\sum_{i_{0 k}\in I_0, i >m } {\tilde v}_k Y^0_{i_{0 k}}
+ \sum_{i_{0 k}\in J_0}{\bar v}_k   \frac{{ \e}^{d(i_{0,k})-1} }{\bar \e^{d(i_{0,k})-1} }Y^{\bar \e}_{i_{0 k}}
\Big)(x)$$

This function is defined and invertible for
$$|\hat v_k |, \ | \tilde v_k| ,\ |{\bar v}_k  | \frac{{ \e}^{d(i_{0,k})-1} }{\bar \e^{d(i_{0,k})-1} }\leq C_{2, \bar \e}r^{d_{\bar \e}(i_{\bar \e k})},  |u_j |\frac{{ \e}^{d(i_{0,j})-1} }{\bar \e^{d(i_{0,j})-1} }\leq C_{1, \bar \e} r^{d_{\bar \e}(i_{\bar \e j})}.$$
 Recall that with the present choice of $r<\e<\bar \e$, we have $C_{1, \bar \e} r^{d_{\bar \e}(i_{\bar \e j})}=C_{1, \bar \e} r^{d_{ \e}(i_{\bar \e j})}=C_{1, \bar \e} r^{d_{ \e}(i_{\e j})}$.
If we set
$$|\hat v_k|,| \tilde v_k|,  |\bar v_k |\leq C_{2, \bar \e} r^{d_{\bar \e(i_{\bar \e k})}},
$$$$|u_j|\leq C_{1, \bar \e} r^{d_{\bar \e}  (i_{\bar \e j})},$$

and argue similarly to Case 1, then the function $\Phi_{\e, v, x}$ will satisfy conditions i), ii), and iii) on $Q(C_{1,\bar \e}r)$  and hence on $Q(C_{1,\e}r)$, with constants independent of $\e$. \end{proof}

\bigskip

We now turn our attention to the Poincar\'{e} inequality and prove that it holds
 with  constant independent of $\epsilon$.  Our argument rests on a rather direct proof from \cite{LanMor} which in some respects simplifies the method used by Jerison in \cite{jer:poincare}. Using some Jacobian estimates from
  \cite{GarofaloNhieu:lip1998} or  \cite{frss:embedding} we will establish that the assumptions required in the key result \cite[Theorem 2.1]{LanMor} are satisfied independently from $\e\ge 0$.  We start by recalling

\begin{thrm}\cite[Theorem 2.1]{LanMor} \label{LM}
 Assume that the doubling condition $(D)$ is satisfied and there exist a sphere $B_\epsilon(x_0, r)$, a cube $Q_\epsilon \subset \R^n$
and a map $E: B_\e(x_0, r)\times Q_\epsilon \rightarrow \R^n$ satisfying the following conditions:
\begin{itemize}
\item [i)]    $B_\e(x_0, 2 r)\subset E(x, Q_\epsilon)$ \quad for every $x\in B_\e(x_0, r)$
\item[ii)]the function $u \mapsto E(x, u)$ is one to one on the box $ Q_\epsilon$ as a function of the variable $u$ and there exists a constant  $\alpha_1>0$ such that $$\frac{1}{\alpha_1} |JE(x,0)| \leq |JE(x,u)| \leq \alpha_1 |JE(x,0)| \quad  \text{ for every  } u \in Q_\epsilon$$
\end{itemize}
Also assume that there exists a positive constant $\alpha_2$, and a function $\gamma:  B_\e(x_0, r) \times Q_\epsilon \times [0,\alpha_2 r]\rightarrow \R^n$
satisfying the following conditions
\begin{itemize}
 \item[iii)]  For every $(x,u) \in  B_\e(x_0, r) \times Q_\epsilon$ the function
$t \mapsto \gamma(x,u,t)$ is a subunit path connecting $x$ and $E(x,u)$
 \item[iv)] For every $(h,t) \in  B_\e(x_0, r) \times Q_\epsilon$ the function $x \mapsto \gamma(x,u,t)$ is a one-to-one map and there exists a constant $\alpha_3>0$ such that
     $$\inf_{ B_\e(x_0, r)\times Q_\epsilon} \Big|det \frac{\partial \gamma}{\partial x}\Big|\geq\alpha_3$$
\end{itemize}
Then there exists a constant $C_P$ depending only on the constants $\alpha_1, \alpha_2, \alpha_3$ and the doubling constant $C_D$ such that (P) is satisfied.
\end{thrm}

\begin{prop}
The vector fields $(X^\e_i)_{i=1\cdots p}$ satisfy condition (P) with a constant independent of $\e$.
\end{prop}
\begin{proof}
We use Proposition \ref{indipendenza} to  show that the assumptions of Theorem \ref{LM} are satisfied unformly in $\e$.
Fix a value $\e\ge 0$ and a set $K=B_\epsilon(x_0, r)$. Choose the constants $C_i$
as in Proposition \ref{indipendenza} and Theorem \ref{MAINNSWeps}, let $Q_\epsilon=Q_\e(\frac{3 C_{1}}{C_2} r)$ and set
$$E(x,u)= \Phi_{\epsilon, 0,x}(u), \text{ defined on }
K \times Q_\epsilon\rightarrow \R^n.$$

To establish assumption (i) of Theorem \ref{LM} it suffices to note that by virtue of  condition (iii) in Theorem  \ref{MAINNSWeps} one has that for $x\in B_\epsilon(x_0, r)$, $$B_\e(x_0, 2 r)\subset B_\e(x, 3 r)\subset E(x, Q_\epsilon).$$
 Assumption (ii) in Theorem \ref{LM} is a direct consequence of condition (ii) in Theorem  \ref{MAINNSWeps}, with $\alpha_1= 16$. By the classical connectivity result of Chow   we see that $E(x,u)$ satisfies assumption (iii), with a function $\gamma$, piecewise expressed as exponential mappings of vector fields of $\e-$degree one. Let us denote   $(X^\e_i)_{i\in I_\e}$ the required vector fields. With this choice of path, it is known (see for example \cite[Lemma 2.2]{GarofaloNhieu:lip1998} or  \cite[pp 99-101]{frss:embedding}) that $x\rightarrow \gamma(x,u,t)$ is a $C^1$ path, with Jacobian determinant
$$\bigg|det \frac{\partial \gamma}{\partial x}(x,u,t)\bigg|= 1 + \psi(x,u,t),$$
for a suitable function $\psi(x,u,t)$
satisfying $$|\psi(x,u,t)|\leq cr, \text{  on  }K\times Q_\e \times [0,cr].$$
Moreover the constant $c$ only depends on the Lipschitz constant of the vector fields $(X^\e_i)_{i\in I_\e}$. Hence in our setting it can be chosen independently of $\e$. Consequently also condition (iv) is satisfied.

\end{proof}

\bibliographystyle{acm}
\bibliography{papers}

\def\cprime{$'$} \def\cprime{$'$}
\begin{thebibliography}{10}

\bibitem{Aimar}
{\sc Aimar, H.}
\newblock Elliptic and parabolic {BMO} and {H}arnack's inequality.
\newblock {\em Trans. Amer. Math. Soc. 306}, 1 (1988), 265--276.

\bibitem{MR0244638}
{\sc Aronson, D.~G., and Serrin, J.}
\newblock Local behavior of solutions of quasilinear parabolic equations.
\newblock {\em Arch. Rational Mech. Anal. 25\/} (1967), 81--122.

\bibitem{CC}
{\sc Capogna, L., and Citti, G.}
\newblock Generalized mean curvature flow in {C}arnot groups.
\newblock {\em Comm. Partial Differential Equations 34}, 7-9 (2009), 937--956.

\bibitem{ccm1}
{\sc Capogna, L., Citti, G., and Manfredini, M.}
\newblock Regularity of non-characteristic minimal graphs in the {H}eisenberg
  group {$\Bbb H^1$}.
\newblock {\em Indiana Univ. Math. J. 58}, 5 (2009), 2115--2160.

\bibitem{CCM2}
{\sc Capogna, L., Citti, G., and Manfredini, M.}
\newblock Smoothness of lipschitz minimal intrinsic graphs in heisenberg groups
  $\mathbb{H}^n$, $n>1$.
\newblock {\em Crelle's Journal\/} (2010).

\bibitem{ch:lipschitz}
{\sc Cheeger, J.}
\newblock Differentiability of {L}ipschitz functions on metric measure spaces.
\newblock {\em Geom.\ Funct.\ Anal. 9\/} (1999), 428--517.

\bibitem{chy}
{\sc Cheng, J.-H., Hwang, J.-F., and Yang, P.}
\newblock Existence and uniqueness for {$p$}-area minimizers in the
  {H}eisenberg group.
\newblock {\em Math. Ann. 337}, 2 (2007), 253--293.

\bibitem{cittimanfredini:uniform}
{\sc Citti, G., and Manfredini, M.}
\newblock Uniform estimates of the fundamental solution for a family of
  hypoelliptic operators.
\newblock {\em Potential Anal. 25}, 2 (2006), 147--164.

\bibitem{CiMa-F}
{\sc Citti, G., and Manfredini, M.}
\newblock Uniform estimates of the fundamental solution for a family of
  hypoelliptic operators.
\newblock {\em Potential Anal. 25}, 2 (2006), 147--164.

\bibitem{frss:embedding}
{\sc Franchi, B., Serapioni, R., and Serra-Cassano, F.}
\newblock Approximation and imbedding theorems for weighted sobolev spaces
  associated with lipschitz continuous vector fields.
\newblock {\em Boll. Un. Mat. Ital. B 7}, 11 (1997), 83--117.

\bibitem{GarofaloNhieu:lip1998}
{\sc Garofalo, N., and Nhieu, D.~M.}
\newblock Lipschitz continuity, global smooth approximations and exten- sion
  theorems for sobolev functions in carnot-carathodory spaces.
\newblock {\em J. Anal. Math. 74\/} (1998), 67--97.

\bibitem{grig}
{\sc Grigor'yan, A.~A.}
\newblock The heat equation on non-compact riemannian manifolds.
\newblock {\em Mat. Sb. (1) 182\/} (1991).

\bibitem{hak:sobolev2}
{\sc Haj{\l}asz, P., and Koskela, P.}
\newblock Sobolev met {P}oincar\'e.
\newblock {\em Memoirs Amer.\ Math.\ Soc. 145}, 688 (2000).

\bibitem{heinonen:calc}
{\sc Heinonen, J.}
\newblock Nonsmooth calculus.
\newblock {\em Bull. Amer. Math. Soc. (N.S.) 44}, 2 (2007), 163--232.

\bibitem{jer:poincare}
{\sc Jerison, D.}
\newblock The {P}oincar{\'e}\ inequality for vector fields satisfying
  {H}{\"o}rmander's condition.
\newblock {\em Duke Math.\ J. 53\/} (1986), 503--523.

\bibitem{MR2274548}
{\sc Kinnunen, J., and Kuusi, T.}
\newblock Local behaviour of solutions to doubly nonlinear parabolic equations.
\newblock {\em Math. Ann. 337}, 3 (2007), 705--728.

\bibitem{KMMP}
{\sc Kinnunen, J., Marola, N., Miranda, M.~J., and Paronetto, f.}
\newblock Harnack's inequality for parabolic de giorgi classes in metric
  spaces.
\newblock {\em Advances in Differential Equations 17\/} (2012), 801--832.

\bibitem{LanMor}
{\sc Lanconelli, E., and Morbidelli, D.}
\newblock On the poincare inequality for vector fields.
\newblock {\em Ark. Mat. 38\/} (2000), 327--342.

\bibitem{MR1359957}
{\sc Maheux, P., and Saloff-Coste, L.}
\newblock Analyse sur les boules d'un op\'erateur sous-elliptique.
\newblock {\em Math. Ann. 303}, 4 (1995), 713--740.

\bibitem{Montgomery:book}
{\sc Montgomery, R.}
\newblock {\em A tour of sub-{R}iemannian geometries, their geodesics and
  applications.}
\newblock No.~91 in Mathematical Surveys and Monographs. American Mathematical
  Society, 2002.

\bibitem{monti-tesi}
{\sc Monti, R.}
\newblock Distances, boundaries and surface measures in carnot-caratheodory
  spaces.
\newblock {\em Ph.D Thesis, Universit\'a degli studi di Trento\/} (2001).

\bibitem{NSW}
{\sc Nagel, A., Stein, E.~M., and Wainger, S.}
\newblock Balls and metrics defined by vector fields. {I}. {B}asic properties.
\newblock {\em Acta Math. 155}, 1-2 (1985), 103--147.

\bibitem{Pauls:minimal}
{\sc Pauls, S.~D.}
\newblock Minimal surfaces in the {H}eisenberg group.
\newblock {\em Geom.\ Dedicata 104\/} (2004), 201--231.

\bibitem{rea}
{\sc Rea, G.~J.}
\newblock A harnack inequality and holder continuity for weak solutions to
  parabolic operators involving hormander vector fields.
\newblock {\em Preprint arxiv 1010.1554v1\/} (2010).

\bibitem{SC}
{\sc Saloff-Coste, L.}
\newblock A note on {P}oincar\'e, {S}obolev, and {H}arnack inequalities.
\newblock {\em Internat. Math. Res. Notices 1992}, 2 (1992), 27--38.

\bibitem{SCS1991}
{\sc Saloff-Coste, L., and Stroock, D.~W.}
\newblock Op\'erateurs uniform\'ement sous-elliptiques sur les groupes de
  {L}ie.
\newblock {\em J. Funct. Anal. 98}, 1 (1991), 97--121.

\bibitem{ste:harmonic}
{\sc Stein, E.~M.}
\newblock {\em Harmonic analysis: real-variable methods, orthogonality, and
  oscillatory integrals}, vol.~43 of {\em Princeton Mathematical Series}.
\newblock Princeton University Press, Princeton, NJ, 1993.
\newblock With the assistance of Timothy S. Murphy, Monographs in Harmonic
  Analysis, III.

\bibitem{MR1387522}
{\sc Sturm, K.~T.}
\newblock Analysis on local {D}irichlet spaces. {III}. {T}he parabolic
  {H}arnack inequality.
\newblock {\em J. Math. Pures Appl. (9) 75}, 3 (1996), 273--297.

\end{thebibliography}
\end{document}